\documentclass[final]{modified-elsarticle}
\usepackage{amsmath}
\usepackage{mathrsfs}
\usepackage{amssymb}
\usepackage[amsmath,thmmarks]{ntheorem}
\usepackage{xcolor}
\usepackage{enumerate}
\usepackage{enumitem}
\setlist[enumerate,1]{label=(\arabic*).,font=\textup,
leftmargin=7mm,labelsep=1.5mm,topsep=0mm,itemsep=-0.8mm}
\setlist[enumerate,2]{label=(\alph*).,font=\textup,
leftmargin=7mm,labelsep=1.5mm,topsep=-0.8mm,itemsep=-0.8mm}

\usepackage{geometry}
\usepackage[colorlinks=true,linkcolor=blue,citecolor=blue]{hyperref}

\theoremstyle{plain}
\newtheorem{theorem}{Theorem}

\newtheorem{conjecture}[theorem]{Conjecture}

\newtheorem{lemma}[theorem]{Lemma}

\theorembodyfont{\normalfont}
\newtheorem{definition}[theorem]{Definition}
\newtheorem{remark}[theorem]{Remark}

\theoremstyle{nonumberplain}
\theoremheaderfont{\it}
\theorembodyfont{\normalfont}
\theoremsymbol{\mbox{$\Box$}}
\newtheorem{proof}{Proof\,.}

\begin{document}

\begin{frontmatter}
\title{Some results on the spectral radii of uniform hypergraphs \tnoteref{titlenote}}
\tnotetext[titlenote]{This work was supported by the Hong Kong Research Grant Council
(Grant Nos. PolyU 502111, 501212, 501913 and 15302114) and NSF of China (Grant Nos.
11101263, 11471210) and by a grant of ``The First-class Discipline of Universities in Shanghai".}

\author[address1]{Liying Kang}

\author[address1]{Lele Liu}

\author[address2]{Liqun Qi}

\author[address1]{Xiying Yuan\corref{correspondingauthor}}
\cortext[correspondingauthor]{Corresponding author}
\ead{xiyingyuan2007@hotmail.com}

\address[address1]{Department of Mathematics, Shanghai University, Shanghai 200444, China}
\address[address2]{Department of Applied Mathematics, The Hong Kong Polytechnic
University, HungHom, Kowloon, HongKong}

\begin{abstract}
Let $\mathcal{A}(G)$ be the adjacency tensor (hypermatrix) of uniform hypergraph $G$. The maximum
modulus of the eigenvalues of $\mathcal{A}(G)$ is called the spectral radius of $G$. In this
paper, the conjecture of Fan et al. in \cite{FanTan2015} related to compare the spectral radii of
some three uniform hypergraphs is solved. Moreover, some eigenvalues properties of a kind of uniform
hypergraphs are obtained.
\end{abstract}

\end{frontmatter}

\section{Introduction}

Denote the set $\{1,2,\cdot\cdot\cdot,n\}$ by $[n].$ Hypergraph is a natural
generalization of ordinary graph (see \cite{Berge:Hypergraph}). A {\em hypergraph}
$G=(V(G),E(G))$ on $n$ vertices is a set of vertices, say $V(G)=\{1,2,\cdots,n\}$
and a set of edges, say $E(G)=\{e_{1},e_{2},\cdots,e_{m}\}$, where
$e_{i}=\{i_{1},i_{2},\cdots,i_{l}\},i_{j}\in\lbrack n]$, $j=1$, $2$, $\cdots$, $l$.
If $|e_{i}|=k$ for any $i=1$, $2$, $\cdots$, $m$, then $G$ is called a {\em $k$-uniform}
hypergraph. In particular, the 2-uniform hypergraphs are exactly the ordinary graphs.
For a vertex $v\in V(G)$ the {\em degree} $d_{G}(v)$ is defined as
$d_{G}(v)=|\{e_{i}:v\in e_{i}\in E(G)\}|$. Vertex with degree one is called
{\em pendent vertex} in this paper. Denote by $G-e$ a new graph (hypergraph) obtained
from $G$ by deleting the edge $e$ of $G$, and by $G+e$ a new graph (hypergraph) obtained
from $G$ by adding the edge $e$ with  $e \not \in E(G)$.

An order $k$ dimension $n$ tensor $\mathcal{T=}(\mathcal{T}_{i_{1}i_{2}\cdots
i_{k}})\in\mathbb{C}^{n\times n\times\cdots\times n}$ over the complex field
$\mathbb{C}$ is a multidimensional array with $n^{k}$ entries, where
$i_{j}\in\lbrack n]$ for each $j=1$, $2$, $\cdots$, $k$.

To study the properties of uniform hypergraphs by algebraic methods,
adjacency matrix has been generalized to adjacency tensor (hypermatrix)
in \cite{Cooper:Spectra Uniform Hypergraphs}.

\begin{definition}
[\cite{Cooper:Spectra Uniform Hypergraphs}]
Let $G=(V(G),E(G))$ be a $k$-uniform hypergraph on $n$ vertices. The adjacency
tensor of $G$ is defined as the $k$-th order $n$-dimensional tensor $\mathcal{A}(G)$
whose $(i_{1}\cdots i_{k})$-entry is
\[
(\mathcal{A}(G))_{i_{1}i_{2}\cdots i_{k}}=%
\begin{cases}
\frac{1}{(k-1)!} & \text{if}~\{i_{1},i_{2},\cdots,i_{k}\}\in E(G),\\
0 & \text{otherwise}.
\end{cases}
\]
\end{definition}

\begin{definition}
[\cite{Qi2005}]
Let $\mathcal{T}$ be an order $k$ dimension $n$ tensor, $x=(x_{1},%
\cdots,x_{n})^{T}\in\mathbb{C}^{n}$ be a column vector of dimension $n$. Then
$\mathcal{T}x^{k-1}$ is defined to be a vector in $\mathbb{C}^{n}$ whose $i$-th
component is the following
\begin{equation}
(\mathcal{T}x^{k-1})_{i}=\sum_{i_{2},\cdots,i_{k}=1}^{n}\mathcal{T}_{ii_{2}\cdots i_{k}%
}x_{i_{2}}\cdots x_{i_{k}}, \quad (i=1,\cdots,n). \label{Formular for Ax}%
\end{equation}
Let $x^{[r]}=(x_{1}^{r},\cdots,x_{n}^{r})^{T}$. Then a number $\lambda\in\mathbb{C}$
is called an {\em eigenvalue} of the tensor $\mathcal{T}$ if there exists a nonzero
vector $x\in\mathbb{C}^{n}$ such that
\begin{equation}
\label{Eigenequations}
\mathcal{T}x^{k-1}=\lambda x^{[k-1]}.
\end{equation}
and in this case, $x$ is called an {\em eigenvector} of $\mathcal{T}$ corresponding
to the eigenvalue $\lambda$.
\end{definition}

By using the general product of tensors defined in \cite{Shao2013},
$\mathcal{T}x^{k-1}$ can be simply written as $\mathcal{T}x$. In the remaining
part of this paper, we will use $\mathcal{T}x$ to denote $\mathcal{T}x^{k-1}$.

In \cite{Friedland:Perron-Frobenius Theorems}, the weak irreducibility of nonnegative
tensors was defined. It was proved in \cite{Friedland:Perron-Frobenius Theorems} and
\cite{Yang:Nonegative Weakly Irreducible Tensors} that a $k$-uniform hypergraph $G$
is connected if and only if its adjacency tensor $\mathcal{A}(G)$ is weakly irreducible.

The spectral radius of $\mathcal{T}$ is defined as
$\rho(\mathcal{T})=\max\{|\lambda|:\lambda~\text{is an eigenvalue of}~\mathcal{T}\}$.
Part of Perron-Frobenius theorem for nonnegative tensors is stated in the following
for reference.

\begin{theorem}
[\cite{K.C.Chang.etc:Perron-Frobenius Theorem},\cite{Yang:Nonegative Weakly Irreducible Tensors}]
\label{Perron-Frobenius}
Let $\mathcal{T}$ be a nonnegative tensor. Then we have the following statements.

\begin{enumerate}
\item $\rho(\mathcal{T})$ is an eigenvalue of $\mathcal{T}$ with a nonnegative
eigenvector $x$ corresponding to it.

\item If $\mathcal{T}$ is weakly irreducible, then $x$ is positive, and for
any eigenvalue $\mu$ with nonnegative eigenvector, $\mu=\rho(\mathcal{T})$ holding.

\item The nonnegative eigenvector $x$ corresponding to $\rho(\mathcal{T})$ is unique
up to a constant multiple.
\end{enumerate}
\end{theorem}

For weakly irreducible nonnegative $\mathcal{T}$ of order $k$, the positive
eigenvector $x$ with $||x||_{k}=1$ corresponding to $\rho(\mathcal{T})$ is
called the {\em principal eigenvector} of $\mathcal{T}$ in this paper.

\section{Comparison of spectral radii of $B_{m}^{L}(1)$, $B_{m}^{L}(2)$ and $B_{m}^{P}$}

For any two edges $e_{i}$ and $e_{j}$ of hypergraph $G$, if
$|e_{i}\cap e_{j}|\leqslant1$, $i\neq j$, then $G$ is called a
{\em linear hypergraph} (see \cite{Bretto}). Let $G$ be a
connected $k$-uniform hypergraph with $n$ vertices and $m$
edges. Then $G$ is called a {\em bicyclic hypergraph} if
$m(k-1)-n=1$ holding (see \cite{FanTan2015}).

\begin{figure}
\begin{minipage}{0.33\textwidth}
\centering
\includegraphics[scale=0.7]{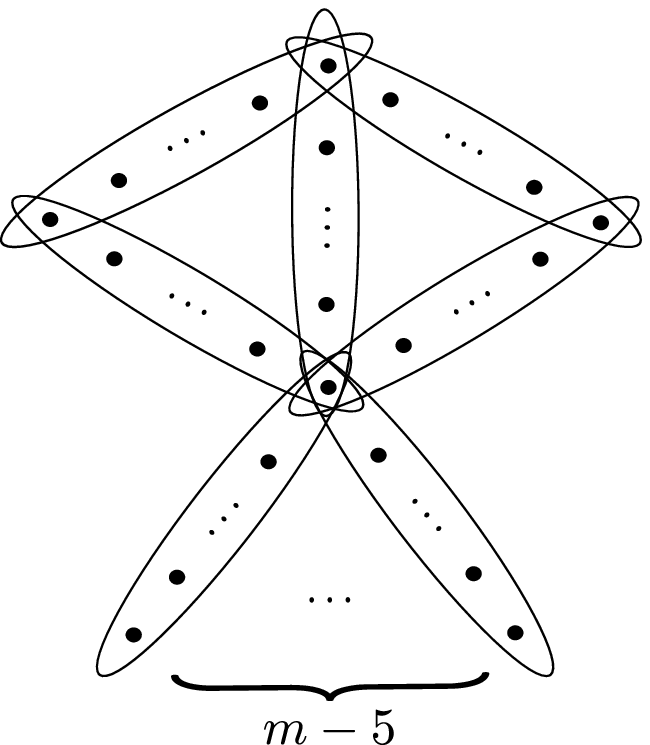}\\[1mm]
(1)
\end{minipage}
\begin{minipage}{0.33\textwidth}
\centering
\includegraphics[scale=0.7]{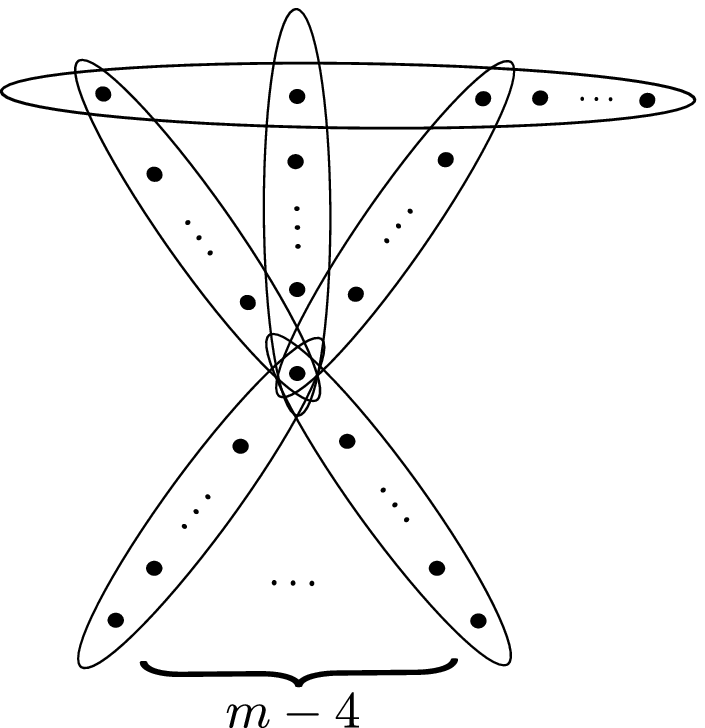}\\[1mm]
(2)
\end{minipage}
\begin{minipage}{0.33\textwidth}
\centering
\includegraphics[scale=0.8]{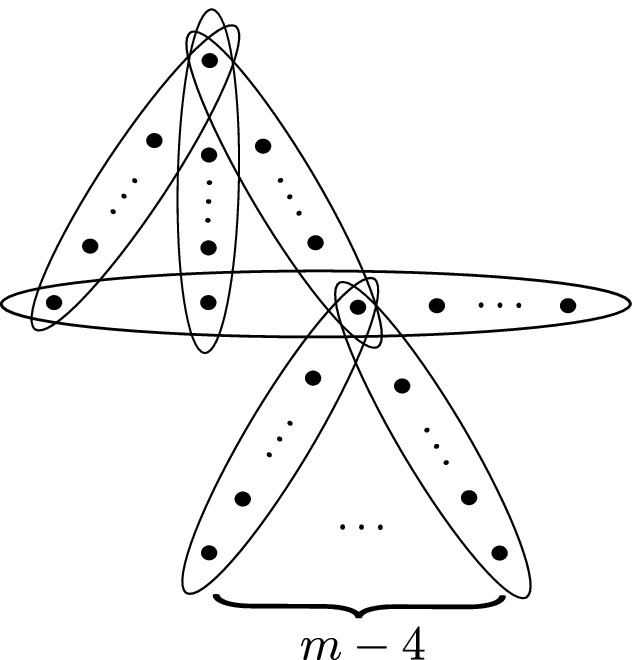}\\[1mm]
(3)
\end{minipage}
\caption{(1) $B_m^P$~~~(2) $B_m^L(1)$~~~(3) $B_m^L(2)$}
\label{Fig:B_m^P}
\end{figure}

Let $B_{m}^{L}(1)$, $B_{m}^{L}(2)$ and $B_{m}^{P}$ be $k$-uniform hypergraphs with
$m$ edges as shown in Figure \ref{Fig:B_m^P}. Theorem 3.9 of \cite{FanTan2015} stated
that among all the linear bicyclic uniform hypergraphs with $m\geqslant5$ edges, the
hypergraph maximizing the spectral radius is among one the three hypergraphs:
$B_{m}^{L}(1)$, $B_{m}^{L}(2)$ and $B_{m}^{P}$. For further information the
following conjecture was presented.

\begin{conjecture}
[\cite{FanTan2015}]
\label{conjecture}
For $m\geqslant5$, $\rho\left(B_{m}^{L}(1)\right) >\rho\left(B_{m}^{L}(2)\right)
>\rho\left( B_{m}^{P}\right)$.
\end{conjecture}

A novel method, weighted incidence matrix method is introduced by Lu and Man
in \cite{LuAndMan:Small Spectral Radius}. It should be announced that spectral
radius defined in \cite{LuAndMan:Small Spectral Radius} differ from this paper,
while for a $k$-uniform hypergraph $G$ the spectral radius defined in
\cite{LuAndMan:Small Spectral Radius} equals to $(k-1)!\rho(G)$, and then it
does not effect the result. We will use this method to prove
$\rho\left(B_{m}^{L}(1)\right) >\rho\left(B_{m}^{L}(2)\right)$.

\begin{definition}
[\cite{LuAndMan:Small Spectral Radius}]
\label{defn:WeightedIncidenceMatrix}
A {\em weighted incidence matrix} $B$ of a hypergraph $H=(V(H),E(H))$ is a
$|V(H)|\times |E(H)|$ matrix such that for any vertex $v$ and any edge $e$,
the entry $B(v,e)>0$ if $v\in e$ and $B(v,e)=0$ if $v\notin e$.
\end{definition}

\begin{definition}
[\cite{LuAndMan:Small Spectral Radius}]~
\begin{enumerate}
\item A hypergraph $H$ is called {\em $\alpha$-normal} if there exists a
weighted incidence matrix $B$ satisfying \par

(i). $\sum_{e:v\in e}B(v,e)=1$, for any $v\in V(H)$; \par
(ii). $\prod_{v:v\in e}B(v,e)=\alpha$, for any $e\in E(H)$. \par
Moreover, the weighted incidence matrix $B$ is called {\em consistent} if
for any cycle $v_{0}e_{1}v_{1}e_{2}\cdots v_{\ell}$ $(v_{\ell}=v_{0})$
\[
\prod^{\ell}_{i=1}\frac{B(v_{i},e_{i})}{B(v_{i-1},e_{i})}=1.
\]
In this case, $H$ is called {\em consistently $\alpha$-normal}.

\item A hypergraph $H$ is called {\em $\alpha$-subnormal} if there exists a
weighted incidence matrix $B$ satisfying \par

(i). $\sum_{e:v\in e}B(v,e)\leqslant 1$, for any $v\in V(H)$; \par
(ii). $\prod_{v: v\in e}B(v,e)\geqslant\alpha$, for any $e\in E(H)$.\par
Moreover, $H$ is called {\em strictly $\alpha$-subnormal} if it is
$\alpha$-subnormal but not $\alpha$-normal.
\end{enumerate}
\end{definition}

\begin{lemma}
[\cite{LuAndMan:Small Spectral Radius}]
\label{lem:StrictlySubnormal}
Let $H$ be a connected $k$-uniform hypergraph.
\begin{enumerate}
\item $H$ is consistently $\alpha$-normal if and only if $\rho(H)=\alpha^{-\frac{1}{k}}$.\label{item1}
\item If $H$ is $\alpha$-subnormal, then $\rho(H)\leqslant\alpha^{-\frac{1}{k}}$.
Moreover, if $H$ is strictly $\alpha$-subnormal, then $\rho(H)<\alpha^{-\frac{1}{k}}$.
\end{enumerate}
\end{lemma}

\begin{theorem}
If $m\geqslant 5$, then $\rho\left(B_{m}^{L}(1)\right)>\rho\left(B_{m}^{L}(2)\right)$.
\end{theorem}	

\begin{proof}
Let
\[
f(x)=(m-4)x^4-(m-1)x^3-x+1.
\]

It is claimed that there exists a unique zero of $f(x)$ in interval $(0,1)$
whenever $m\geqslant 5$. In fact, it is clear that $f(0)=1>0$, $f(1)=-3<0$
and $f(+\infty)=+\infty$, therefore $f(x)$ has at least two real zeros located
in $(0,1)$ and $(1,+\infty)$. Suppose that $f(x)$ has four real zeros $x_1$,
$x_2$, $x_3$, $x_4$. Notice that $f(x)>0$ whenever $x\leqslant 0$, it follows
that all the real zeros of $f(x)$ are located in $(0,+\infty)$. Hence $x_i>0$,
$i=1$, $2$, $3$, $4$. However, according to Vieta's formulas we have
\[
\sum_{1\leqslant i<j\leqslant 4}x_ix_j=0,
\]
which is a contradiction with $x_i>0$, $i=1$, $2$, $3$, $4$. Therefore $f(x)$ has
only two real zeros located in $(0,1)$ and $(1,+\infty)$, respectively. Thus
there exists a unique zero of $f(x)$ in interval $(0,1)$.

Suppose that $\alpha^{\frac13}$ is the zero of $f(x)$ in interval $(0,1)$, i.e.,
\begin{equation}
\label{eq:alpha}
(m-4)\alpha^{\frac43}-(m-1)\alpha-\alpha^{\frac13}+1=0.
\end{equation}
In what follows, we first prove that $\alpha^{\frac13}$ is monotonically decreasing
on $m$. For convenience, we denote by $y=\alpha^{\frac13}$, i.e.,
\begin{equation}
\label{eq:y}
(m-4)y^4-(m-1)y^3-y+1=0.
\end{equation}
Take the derivative of both sides of \eqref{eq:y} on $m$, we have
\begin{equation}
\label{eq:y'}
[4(m-4)y^3-3(m-1)y^2-1]\cdot y'=y^3-y^4.
\end{equation}
From \eqref{eq:y}, we obtain
\[
m=\frac{4y^4-y^3+y-1}{y^4-y^3}.
\]
Substituting this into \eqref{eq:y'}, we see
\[
y'=-\frac{y^4(y-1)^2}{3[y^4+(y-1)^2]}<0,
\]
which implies that $\alpha^{\frac13}$ is monotonically decreasing on $m$. Hence
\begin{equation}
\label{eq:alpha>=<1/2}
\begin{cases}
\alpha^{\frac13}>\frac12 & \text{if}~m=5,\\
\alpha^{\frac13}=\frac12 & \text{if}~m=6,\\
\alpha^{\frac13}<\frac12 & \text{if}~m\geqslant 7.
\end{cases}
\end{equation}

We now prove that $B_m^L(1)$ is consistently $\alpha$-normal. Label edges
and vertices of $B_m^L(1)$ as shown in Figure \ref{labeled B-m-L(1)}.

\begin{figure}
\begin{minipage}{0.5\textwidth}
\centering
\includegraphics[scale=0.7]{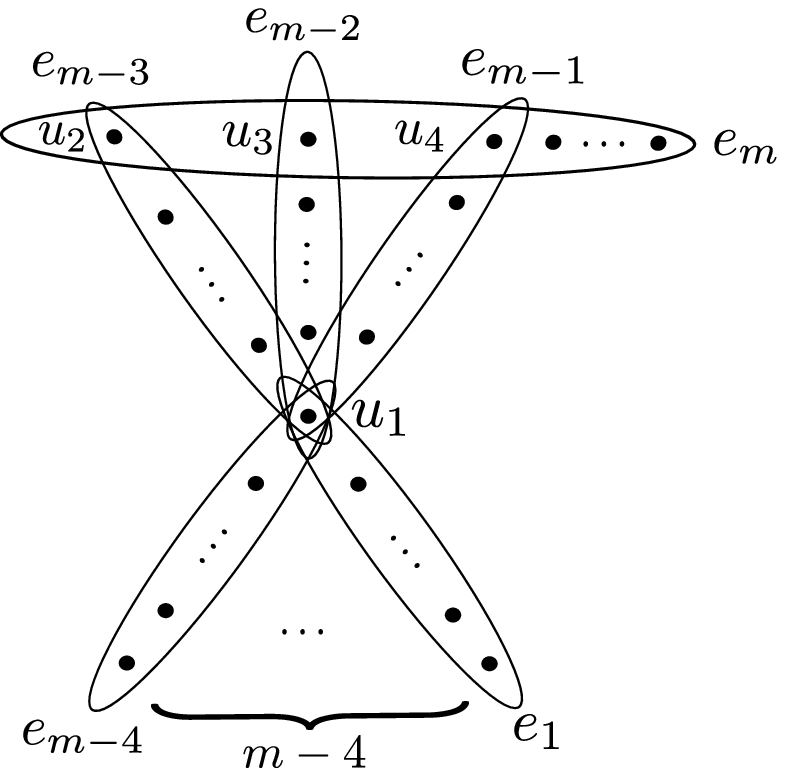}
\caption{$B_m^L(1)$}
\label{labeled B-m-L(1)}
\end{minipage}
\begin{minipage}{0.5\textwidth}
\centering
\includegraphics[scale=0.75]{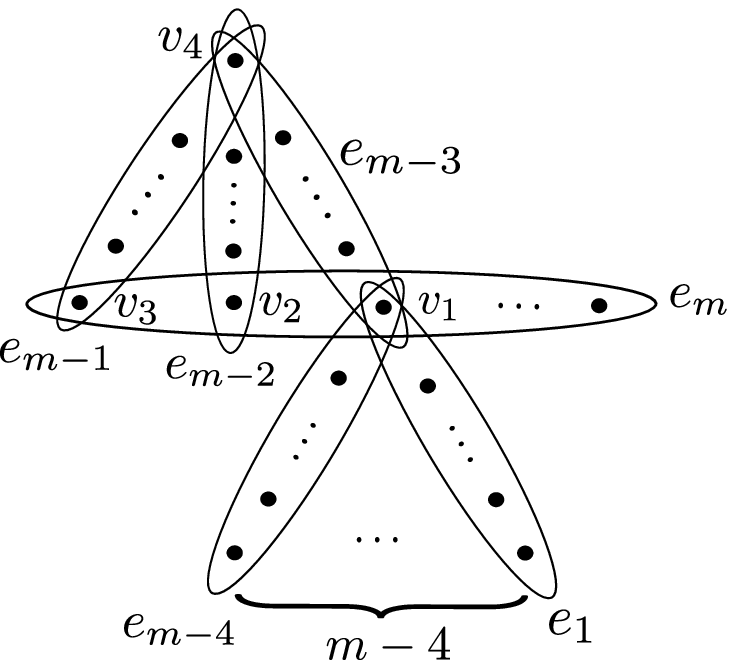}
\caption{$B_m^L(2)$}
\label{labeled B-m-L(2)}
\end{minipage}
\end{figure}

Construct a weighted incidence matrix $B(u,e)$ of $B_m^L(1)$ as follows.
\[
B(u,e)=
\begin{cases}
0 & u\notin e,\\
1 & u~\text{is a pendent vertex in}~e,\\
\alpha & (u,e)=(u_1,e_i), i=1,2,\cdots,m-4,\\[2mm]
\displaystyle\frac{\alpha}{1-\alpha^{\frac13}} & (u,e)=(u_1,e_i),i=m-3,m-2,m-1,\\[3mm]
\displaystyle 1-\alpha^{\frac13} & (u,e)=(u_2,e_{m-3}),(u_3,e_{m-2}),(u_4,e_{m-1}),\\[2mm]
\displaystyle\alpha^{\frac13} & (u,e)=(u_i,e_m),i=2,3,4.
\end{cases}
\]
It can be checked that
\[
\sum_{e:u\in e}B(u,e)=1,~~\text{for any}~u\in V(B_m^L(1)),
\]
\[
\prod_{u:u\in e}B(u,e)=\alpha,~~\text{for any}~e\in E(B_m^L(1)),
\]
which yields that $B_m^L(1)$ is $\alpha$-normal.
For cycle $u_1e_{m-3}u_2e_mu_3e_{m-2}u_1$, we have
\[
\frac{B(u_2,e_{m-3})}{B(u_1,e_{m-3})}\cdot
\frac{B(u_3,e_{m})}{B(u_2,e_{m})}\cdot
\frac{B(u_1,e_{m-2})}{B(u_3,e_{m-2})}=
\frac{1-\alpha^{\frac13}}{\frac{\alpha}{1-\alpha^{\frac13}}}\cdot
\frac{\alpha^{\frac13}}{\alpha^{\frac13}}\cdot
\frac{\frac{\alpha}{1-\alpha^{\frac13}}}{1-\alpha^{\frac13}}=1.
\]
Similarly, for cycles $u_1e_{m-3}u_2e_mu_4e_{m-1}u_1$ and $u_1e_{m-2}u_3e_mu_4e_{m-1}u_1$,
we have the same results. Hence $B_m^L(1)$ is consistently $\alpha$-normal, and
therefore $\rho(B_m^L(1))=\alpha^{-\frac1k}$ from (1) of Lemma \ref{lem:StrictlySubnormal}.

Now we consider the hypergraph $B_m^L(2)$. Label the edges and vertices of $B_m^L(2)$
as shown in Figure \ref{labeled B-m-L(2)}. We distinguish two cases as follows.\\
{\bf Case 1:} $m=5$.\\
In this case, \eqref{eq:alpha} can be written as
\begin{equation}
\label{eq:alpha m=5}
\alpha^{\frac43}-4\alpha-\alpha^{\frac13}+1=0.
\end{equation}
Construct a weighted incidence matrix $B(v,e)$ of $B_m^L(2)$ as follows.
\[
B(v,e)=
\begin{cases}
0 & v\notin e,\\
1 & v~\text{is a pendent vertex in}~e,\\
\displaystyle \alpha & (v,e)=(v_1,e_1),\\[2mm]
\displaystyle 1-\alpha-\alpha^{\frac13} & (v,e)=(v_1,e_{2}),\\[2mm]
\displaystyle \alpha^{\frac13} & (v,e)=(v_i,e_5),i=1,2,3,\\[1mm]
\displaystyle 1-\alpha^{\frac13} & (v,e)=(v_2,e_3),(v_3,e_4),\\[2mm]
\displaystyle \frac{\alpha}{1-\alpha^{\frac13}} & (v,e)=(v_4,e_i),i=3,4,\\[3mm]
\displaystyle 1-\frac{2\alpha}{1-\alpha^{\frac13}} & (v,e)=(v_4,e_2).
\end{cases}
\]
It can be checked that
\begin{align*}
\sum_{e:v\in e}B(v,e)=1, &~~\text{for any}~v\in V(B_m^L(2)),\\[2mm]
\prod_{v:v\in e}B(v,e)=\alpha, &~~\text{for any}~e\neq e_{2}.
\end{align*}
For the edge $e_2$. It follows from \eqref{eq:alpha>=<1/2} and \eqref{eq:alpha m=5} that
\begin{align*}
\prod_{v:v\in e_{2}}B(v,e_{2})
& =\left(1-\frac{2\alpha}{1-\alpha^{\frac13}}\right)\cdot(1-\alpha-\alpha^{\frac13})\\
& =\frac{[(1-\alpha^{\frac13})-2\alpha]\cdot[(1-\alpha^{\frac13})-\alpha]}{1-\alpha^{\frac13}}\\
& =\frac{[(4\alpha-\alpha^{\frac43})-2\alpha]\cdot[(4\alpha-\alpha^{\frac43})-\alpha]}{4\alpha-\alpha^{\frac43}}\\
& =\alpha+\frac{\alpha(\alpha^{\frac23}-4\alpha^{\frac13}+2)}{4-\alpha^{\frac13}}\\
&>\alpha,
\end{align*}
which yields that $B_m^L(2)$ is strictly $\alpha$-subnormal,
and then by (2) of Lemma \ref{lem:StrictlySubnormal} we have
$\rho(B_m^L(2))<\alpha^{-\frac1k}=\rho(B_m^L(1))$.\\
{\bf Case 2:} $m\geqslant 6$.\\
We construct a weighted incidence matrix $B(v,e)$ for the hypergraph $B_m^L(2)$ as follows.
\[
B(v,e)=
\begin{cases}
0 & v\notin e,\\
1 & v~\text{is a pendent vertex in}~e,\\
\displaystyle \alpha & (v,e)=(v_1,e_i), i=1,2,\cdots,m-4,\\[3mm]
\displaystyle \frac{\alpha}{1-\alpha^{\frac13}} & (v,e)=(v_1,e_{m-3}),\\[3mm]
\displaystyle \frac{2\alpha}{1-\alpha^{\frac13}} & (v,e)=(v_1,e_m),\\[3mm]
\displaystyle 1-2\alpha^{\frac23} & (v,e)=(v_i,e_m), i=2,3,\\[2mm]
\displaystyle 2\alpha^{\frac23} & (v,e)=(v_2,e_{m-2}),(v_3,e_{m-1}),\\[1mm]
\displaystyle \frac{\alpha^{\frac13}}{2} & (v,e)=(v_4,e_i),i=m-2,m-1,\\[3mm]
\displaystyle 1-\alpha^{\frac13} & (v,e)=(v_4,e_{m-3}).
\end{cases}
\]
It can be checked that
\begin{align*}
\sum_{e:v\in e}B(v,e)=1, & ~~\text{for any}~v\in V(B_m^L(2)),\\
\prod_{v:v\in e}B(v,e)=\alpha, & ~~\text{for any}~e\neq e_{m}.
\end{align*}
For the edge $e_{m}$, we have
\begin{align*}
\prod_{v:v\in e_{m}}B(v,e_{m}) & =\frac{2\alpha}{1-\alpha^{\frac13}}\cdot
\left(1-2\alpha^{\frac23}\right)^2\\
& =\frac{2\alpha\left(4\alpha^{\frac43}-4\alpha^{\frac23}+1\right)}{1-\alpha^{\frac13}}\\
& =\alpha+\frac{\alpha(2\alpha^{\frac13}-1)(\alpha^{\frac13}+1)(4\alpha^{\frac23}-2\alpha^{\frac13}-1)}
{1-\alpha^{\frac13}}.
\end{align*}

(i). If $m=6$, then $\alpha^{\frac13}=\frac12$. Therefore
\[
\prod_{v:v\in e_{m}}B(v,e_{m})=\alpha,
\]
which implies that $B_m^L(2)$ is $\alpha$-normal. Then $B_m^L(2)$ is $\alpha$-subnormal and
by (2) of Lemma \ref{lem:StrictlySubnormal} we have $\rho(B_m^L(2))\leqslant \alpha^{-\frac1k}$.
Now we will show $\rho(B_m^L(2))\neq \alpha^{-\frac1k}$. For cycle $v_1e_mv_2e_{m-2}v_4e_{m-3}v_1$,
we have
\[
\frac{B(v_2,e_m)}{B(v_1,e_m)}\cdot
\frac{B(v_4,e_{m-2})}{B(v_2,e_{m-2})}\cdot
\frac{B(v_1,e_{m-3})}{B(v_4,e_{m-3})}=
\frac{1-2\alpha^{\frac23}}{8\alpha^{\frac13}(1-\alpha^{\frac13})}\neq 1.
\]
Therefore matrix $B(v,e)$ is not consistent. Then $B_m^L(2)$ is not
consistently $\alpha$-normal, by (1) of Lemma \ref{lem:StrictlySubnormal}
we know that $\rho(B_m^L(2))\neq \alpha^{-\frac1k}$. It follows that
$\rho(B_m^L(2))<\alpha^{-\frac1k}=\rho(B_m^L(1))$.

(ii). If $m\geqslant 7$, then $\alpha^{\frac13}<\frac12$ from \eqref{eq:alpha>=<1/2}. Therefore
\[
2\alpha^{\frac13}-1<0,~4\alpha^{\frac23}-2\alpha^{\frac13}-1<0.
\]
It follows that
\[
\prod_{v:v\in e_{m}}B(v,e_{m})>\alpha.
\]
That is, $B_m^L(2)$ is strictly $\alpha$-subnormal,
and therefore by (2) of Lemma \ref{lem:StrictlySubnormal} we have
$\rho(B_m^L(2))<\alpha^{-\frac1k}=\rho(B_m^L(1))$.

The proof is completed.
\end{proof}

It is convenient to prove the second part of Conjecture \ref{conjecture} by
using the following expression of spectral radius of a nonegative symmetric tensor.
\begin{theorem}
[\cite{Qi2013}]
\label{relaigh}
Let $\mathcal{T}$ be a nonnegative symmetric tensor of order $k$ and dimension $n$,
denote $\mathbb{R}_{+}^{n}=\{x\in\mathbb{R}^{n}\,|\, x\geqslant 0\}$. Then we have
\begin{equation}
\label{Expression for spectral radius}
\rho(\mathcal{T})=\max\{x^{T}(\mathcal{T}x)\,|\, x\in\mathbb{R}_{+}^{n},
\sum_{i=1}^{n}x_{i}^{k}=1\}.
\end{equation}
Furthermore, $x\in\mathbb{R}_{+}^{n}$ with $\sum_{i=1}^{n}x_{i}^{k}=1$ is an
optimal solution of the above optimization problem if and only
if it is an eigenvector of $\mathcal{T}$ corresponding to the eigenvalue
$\rho(\mathcal{T})$.
\end{theorem}

\begin{theorem}
For $m\geqslant5$ we have $\rho\left(B_{m}^{L}(2)\right) >\rho\left(B_{m}^{P}\right)$.
\end{theorem}

\begin{proof}
Let $x$ be the unit (positive) principal eigenvector of $\mathcal{A}(B_{m}^{P})$,
namely, $\rho\left(B_{m}^{P}\right)=x^{T}\left( \mathcal{A}(B_{m}^{P})x\right)$.
We label some vertices of $B_{m}^{P}$ as shown in Figure \ref{Fig:labeled B-m-P}.
\begin{figure}
\centering
\includegraphics[scale=0.7]{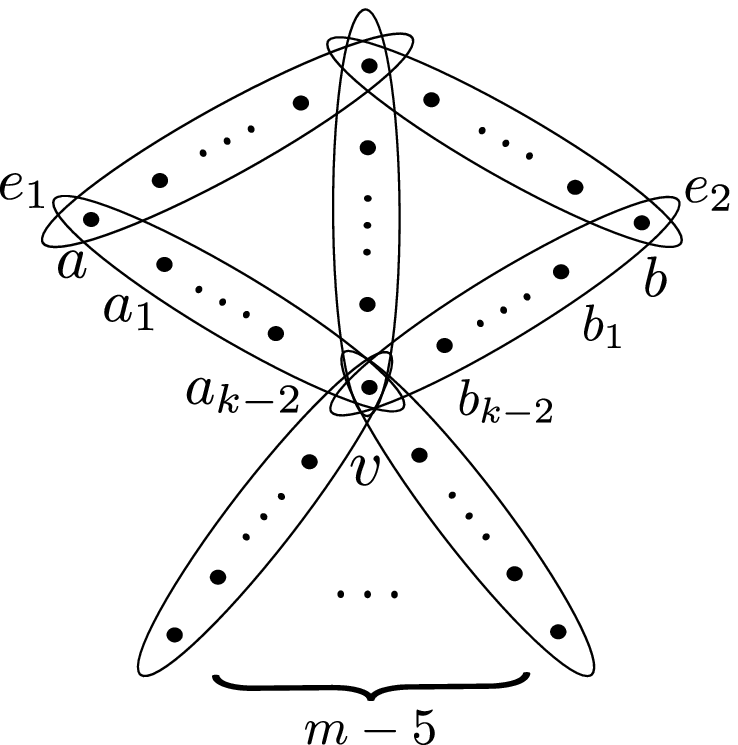}
\caption{$B_m^P$}
\label{Fig:labeled B-m-P}
\end{figure}

By the symmetry of the vertices $a_{1}$, $a_{2}$, $\cdots$, $a_{k-2}$,
we have $x_{a_{1}}=\cdots=x_{a_{k-2}}$. Similarly we have
$x_{b_{1}}=\cdots=x_{b_{k-2}}$, $x_{a}=x_{b}$ and $x_{a_{1}}=x_{b_{1}}$.

We will show  that $x_{a}>x_{a_{1}}$. Write $\rho=\rho(B_{m}^{P})$ for short,
then we have $\rho>(\Delta(B_m^P))^{\frac1k}>0$ (see \cite{Cooper:Spectra Uniform Hypergraphs}),
where $\Delta(B_m^P)$ is the maximum degree of $B_m^P$. By (\ref{Eigenequations}) we have
\[
\begin{cases}
\displaystyle \rho x_{a_{1}}^{k-1}=(x_{a_{1}})^{k-3}x_{a}x_{v},\\[2mm]
\displaystyle \rho x_{a}^{k-1}>(x_{a_{1}})^{k-2}x_{v}.
\end{cases}
\]
Thus $\rho (x_{a_{1}})^{k}=(x_{a_{1}})^{k-2}x_{a}x_{v}<\rho (x_{a})^{k}$,
and then $x_{a}>x_{a_1}$ holding.

Set
\begin{align*}
e_{1}  &  =\left\{ v, a_{1},\cdots, a_{k-2}, a \right\},
\quad~e_{2}=\left\{v, b_{1}, \cdots, b_{k-2},b\right\},\\
e_{1}^{\prime} & =\left\{v,a,b,a_{2},\cdots,a_{k-2}\right\},~~e_{2}^{\prime}
=\left\{v,a_{1},b_{1}, \cdots, b_{k-2} \right\}.
\end{align*}
(If $k=3$, we set $e_{1}^{\prime} =\left\{ v,a,b\right\}$ and
$e_{2}^{\prime}=\left\{v,a_{1},b_{1}\right\}$).
It is obvious that
\[
B_{m}^{L}(2)=B_{m}^{P}-e_{1}- e_{2}+e_{1}^{\prime}+e_{2}^{\prime}.
\]
By Theorem \ref{relaigh} we have
$\rho\left(B_{m}^{L}(2)\right)\geqslant x^{T}\left( \mathcal{A}(B_{m}^{L}(2))x\right)$.
Hence
\begin{align*}
\rho\left(B_{m}^{L}(2)\right) -\rho\left(B_{m}^{P}\right)
&  \geqslant x^{T}\left(\mathcal{A}(B_{m}^{L}(2))x\right)-
x^{T}\left( \mathcal{A}(B_{m}^{P})x\right) \\
&  =x_{v}x_{a}x_{b}(x_{a_{1}})^{k-3}+x_{v}x_{a_{1}}(x_{b_{1}})^{k-2}
-x_{v}(x_{a_{1}})^{k-2}x_{a}-x_{v}(x_{b_{1}})^{k-2}x_{b}\\
& =x_{v}(x_{a_{1}})^{k-3}\left(x_a-x_{a_1}\right)^{2}\\
& >0.
\end{align*}
Thus we prove that $\rho\left(B_{m}^{L}(2)\right) >\rho\left(B_{m}^{P}\right)$.
\end{proof}

\section{An eigenvalues property of generalized hypergraph $G^{k,s}$}

The concept of power hypergraphs was introduced in
\cite{ShenglongHu.etc:Cored Hypergraphs}.

\begin{definition}
[\cite{ShenglongHu.etc:Cored Hypergraphs}]
\label{power graph}
Let $G=(V(G),E(G))$ be an ordinary graph. For every $k\geqslant2$, the $k$-th power
of $G$, $G^{k}:=(V(G^{k}),E(G^{k}))$ is defined as the $k$-uniform hypergraph with
the edge set
\[
E(G^{k}):=\{e\cup\{i_{e,1},\cdot\cdot\cdot,i_{e,k-2}\}\,|\,e\in E(G)\}
\]
and the vertex set
\[
V(G^{k}):=V(G)\cup(\cup_{e\in E(G)}\{i_{e,1},\cdots,i_{e,k-2}\}).
\]
\end{definition}

The definition for $k$-th power hypergraph $G^{k}$ has been generalized by
Khan and Fan in \cite{KhanFan2015}.

\begin{definition}
[\cite{KhanFan2015}]
\label{generalized power hypergraph}
Let $G=(V,E)$ be an ordinary graph. For any $k\geqslant3$ and $1\leqslant s\leqslant\frac{k}{2}$.
For each $v\in V$ (and $e\in E$), let $V_{v}$ (and $V_{e}$) be a new vertex set with $s$ (and $k-2s$)
elements such that all these new sets are pairwise disjoint. Then the generalized power of
$G$, denoted by $G^{k,s}$, is defined as the $k$-uniform hypergraph with the vertex set
\[
V(G^{k,s})=\left( \bigcup_{v\in V}V_{v}\right) \bigcup\left( \bigcup_{e\in E}V_{e}\right)
\]
and edge set
\[
E(G^{k,s})=\{V_{u}\cup V_{v}\cup V_{e}:e=\{u,v\}\in E\}.
\]
\end{definition}

We may further generalize the definition of $G^{k,s}$ from a general graph $G$
to a uniform hypergraph $G$ as follows.

\begin{definition}
\label{generalized G-k,s}
Let $G=(V,E)$ be a $t$-uniform hyergraph. For any $k\geqslant t$ and
$1\leqslant s\leqslant \left\lfloor \frac{k}{t}\right\rfloor$. For each
$v\in V$ (and $e\in E$), let $V_{v}$ (and $V_{e}$) be a new vertex set with
$s$ (and $k-ts$) elements such that all these new sets are pairwise disjoint.
Then the generalized power of $G$, denoted by $G^{k,s}$, is defined as the
$k$-uniform hypergraph with the vertex set
\[
V(G^{k,s})=\left(\bigcup_{v\in V}V_{v}\right)\bigcup\left(\bigcup_{e\in E}V_{e}\right)
\]
and edge set
\[
E(G^{k,s})=\{V_{v_{1}}\cup\cdots\cup V_{v_{t}}\cup V_{e}:e=
\{v_{1},v_{2},\cdots,v_{t}\}\in E\}.
\]
\end{definition}

By constructing a new vector, we will prove a result related to the
relationship between $\rho(\mathcal{A}(G))$ and $\rho(\mathcal{A}(G^{k,s}))$.
Suppose that $x$ is an eigenvector of $\mathcal{A}(G)$ corresponding to $\mu$.
For any edge $e$, we will write
\[
x^{e}=\prod_{v\in e}x_{v},
\]
and
\[
x^{e\backslash\{u\}}=\prod_{v\in(e\backslash\{u\})}x_{v}.
\]
For any vertex $v$ of a $t$-uniform hypergraph $G$, from
\eqref{Formular for Ax} the corresponding eigenequation of
$\mathcal{A}(G)$ with eigenpair $(\mu,x)$ becomes

\begin{equation}
\left( \mathcal{A}(G)x\right)_{v}=\sum_{v\in e}
x^{e\backslash\{v\}}=\mu x_{v}^{t-1}.
\label{eigequations for adjacency tensor}
\end{equation}

\begin{theorem}
\label{thm:generalized power hypergraph}
If $\mu>0$ is an eigenvalue of $\mathcal{A}(G)$ with a nonegative eigenvector,
then $\mu^{\frac{ts}{k}}$ is an eigenvalue of $\mathcal{A}(G^{k,s})$.
Moreover $\rho(\mathcal{A}(G^{k,s}))=\rho(\mathcal{A}(G))^{\frac{ts}{k}}$
for connected hypergraph $G$.
\end{theorem}

\begin{proof}
Suppose that $x$ is a nonnegative eigenvector of the eigenvalue $\mu>0$ of
$\mathcal{A}(G)$. Now we construct a new vector $y$ (of dimension $|V(G^{k,s})|$)
from $x$ by adding components. Set
\[
y_{w}=
\begin{cases}
(x_{v})^{\frac{t}{k}} & \text{if }w\in V_{v}\text{ for some }v,\\
(\mu^{-1}x^{e})^{\frac{1}{k}} & \text{if }w\in V_{e}\text{ for some edge }e.
\end{cases}
\]
Now we will show $\mathcal{A}(G^{k,s})y=\mu^{\frac{ts}{k}}y^{[k-1]}$ holding.

For any $w\in V_{v}$ for some $v$, it follows from
(\ref{eigequations for adjacency tensor}) that
\begin{align}
(\mathcal{A}(G^{k,s})y)_{w}  &  =\sum_{v\in e\in E(G)}[(x_{v})^{\frac{t}{k}%
}]^{s-1}[(x^{e\setminus\{v\}})^{\frac{t}{k}}]^{s}[(\mu^{-1}x^{e})^{\frac{1}%
{k}}]^{k-ts}\nonumber\\
&  =\mu^{\frac{ts}{k}-1}(x_{v})^{\frac{k-t}{k}}\sum_{v\in e\in E(G)}%
x^{e\setminus\{v\}}\nonumber\\
&  =\mu^{\frac{ts}{k}-1}(x_{v})^{\frac{k-t}{k}}\mu(x_{v})^{t-1}\nonumber\\
&  =\mu^{\frac{ts}{k}}[(x_{v})^{\frac{t}{k}}]^{k-1}\nonumber\\
&  =\mu^{\frac{ts}{k}}y_{w}^{k-1}.\nonumber
\end{align}

For $w\in V_{e}$ for any edge $e$, we have
\begin{align}
(\mathcal{A}(G^{k,s})y)_{w}  &  =[(x^{e})^{\frac{t}{k}}]^{s}[(\mu^{-1}%
x^{e})^{\frac{1}{k}}]^{k-ts-1}\nonumber\\
&  =\mu^{\frac{ts}{k}}(\mu^{-1}x^{e})^{\frac{k-1}{k}}\nonumber\\
&  =\mu^{\frac{ts}{k}}y_{w}^{k-1}.\nonumber
\end{align}
Hence $\mu^{\frac{ts}{k}}$ is an eigenvalue of $\mathcal{A}(G^{k,s})$ with
eigenvector $y$.

For connected $t$-uniform hypergraph $G$, choose $x$ as a
positive eigenvector of $\rho(\mathcal{A}(G))$ by Perron-Frobenius theorem for
irreducible nonnegative tensor. In this case $y$ is a positive eigenvector of
the eigenvalue $\rho(\mathcal{A}(G))^{\frac{ts}{k}}$ of tensor $\mathcal{A}(G^{k,s})$.
In virtue of (2) of Theorem \ref{Perron-Frobenius} (or see Lemma 15 of
\cite{J.Zhou.etc:Spectral Properties Uniform Hypergraphs}), we have
\[
\rho(\mathcal{A}(G^{k,s}))=\rho(\mathcal{A}(G))^{\frac{ts}{k}}.
\]

The proof is completed.
\end{proof}

\begin{remark}
The result of Theorem \ref{thm:generalized power hypergraph} generalizes some
known cases as follows.
\begin{enumerate}
\item If we take $t=2$ and $s=1,$ then $G^{k,s}$ becomes the $k$-th power of a
general graph $G$ (see Definition \ref{power graph}), and the corresponding
version of Theorem \ref{thm:generalized power hypergraph} is Theorem 16 of
\cite{J.Zhou.etc:Spectral Properties Uniform Hypergraphs}.
	
\item If we take $t=2$ and $s=\frac{k}{2}$, and the corresponding version of
Theorem \ref{thm:generalized power hypergraph} is the adjacency tensor part
of Lemma 3.12 of \cite{KhanFan2015}.

\item  If we take $t=k-1$ and $s=1$, and the corresponding version of Theorem
\ref{thm:generalized power hypergraph} is a similar result of Lemma 8 in
\cite{LuAndMan:Small Spectral Radius}.

\item If we take $t=2$, then $G^{k,s}$ becomes the generalized power hypergraph
of a general graph (see Definition \ref{generalized power hypergraph}), and
the corresponding version of Theorem \ref{thm:generalized power hypergraph} is
Theorem 21 of \cite{XiyingYuan:conjecture of Qi}.
\end{enumerate}
\end{remark}

From the proof of Theorem \ref{thm:generalized power hypergraph} we obtain the following result.

\begin{lemma}
\label{eigenpair of G-k,s copy(1)}
Let $G$ be a connected $t$-uniform hypergraph with spectral radius $\rho(G)$ and principal
eigenvector $y$. Then $G^{k,s}$ has spectral radius $\rho(G)^{\frac{ts}{k}}$ with an
eigenvector $x$ such that
\[
x_{w}=
\begin{cases}
(y_{v})^{\frac{t}{k}} & \text{if }w\in V_{v}\text{ for some }v\text{ of }G,\\
\left(\frac{y^{e}}{\rho(G)}\right)^{\frac{1}{k}} & \text{if }w\in
V_{e}\text{ for some edge }e\text{ of }G.
\end{cases}
\]
\end{lemma}

\end{document}